\documentclass[preprint,12pt,a4paper,numbers]{elsarticle}
\usepackage[T1]{fontenc}
\usepackage{iftex}
\ifPDFTeX 
  \usepackage[utf8]{inputenc}
\fi
\usepackage{lmodern,microtype}
\usepackage{amsmath,amssymb,amsthm,mathtools}
\usepackage{booktabs,array,longtable}
\usepackage{graphicx,float,etoolbox}
\usepackage[hidelinks]{hyperref}
\usepackage{enumitem}
\setlist{topsep=4pt,itemsep=3pt}
\allowdisplaybreaks[2]
\numberwithin{equation}{section}
\BeforeBeginEnvironment{align}{%
  \begin{subequations}%
}
\AfterEndEnvironment{align}{\end{subequations}}
\newtheorem{theorem}{Theorem}[section]
\newtheorem{lemma}{Lemma}[section]
\newtheorem{proposition}{Proposition}[section]

\theoremstyle{definition}
\newtheorem{definition}{Definition}[section]
\theoremstyle{remark}

\newcommand{\sgn}{\operatorname{sgn}}

\newcommand{\ee}{\mathrm{e}}

\newcommand{\BB}{\mathcal{B}}
\begin{document}
\begin{frontmatter}



	\title{A Sharp Bound of Crossing Limit Cycle for Planar piecewise linear Systems with a Straight Switching Line}


	\author{Zijie Xu} 
	\author{Jiang Yu\corref{author2}} 

	\cortext[author2]{Corresponding author.}
 \ead{jiangyu@sjtu.edu.cn}
	\affiliation{organization={School of mathematical Sciences},
		addressline={Shanghai Jiao Tong University},
		city={Shanghai},
		postcode={200240},
		country={PR China}}
 \begin{abstract}
Determining the maximum number of crossing limit cycles in planar discontinuous piecewise linear systems separated by a straight line is a central problem in nonsmooth dynamical systems and a natural piecewise-linear counterpart of Hilbert's sixteenth problem, yet its sharp uniform upper bound has long remained an open problem. In this paper, we prove that any such system admits at most three crossing limit cycles. Combined with known examples realizing three crossing limit cycles, this result establishes that three is the sharp uniform upper bound.
\end{abstract}
\end{frontmatter}

\section{Introduction}\label{sec:intro}

The second part of Hilbert's sixteenth problem concerns the number of limit cycles of planar polynomial differential systems. More precisely, for each positive integer $n$, it asks whether there exists a uniform upper bound $H(n)$ on the number of limit cycles of such systems of degree $n$. The same question has received considerable attention for planar piecewise linear systems, motivated in part by applications in electrical circuits and mechanics. Although a planar linear system has no limit cycles, a piecewise linear system may have several ones. We consider planar piecewise linear systems with two zones separated by the straight line $\Sigma=\{x=0\}$:
\begin{equation}\label{eq:raw-system}
  \frac{\mathrm{d}\mathbf{z}}{\mathrm{d}t}=
  \begin{cases}
    \mathbf{A}_1\mathbf{z}+\mathbf{B}_1, & x<0,\\
    \mathbf{A}_2\mathbf{z}+\mathbf{B}_2, & x>0,
  \end{cases}
  \qquad \mathbf{z}=(x,y)^{\top}\in\mathbb{R}^2,
\end{equation}
where $\mathbf{A}_i=(a_{ijk})_{2\times2}\in\mathbb{R}^{2\times2},\mathbf{B}_i=(b_{i1},b_{i2})^{\top}\in\mathbb{R}^2,i=1,2$. Write $\Sigma=\{x=0\}$ and $\Omega_\pm=\{\pm x>0\}$. The normal vectors on $\Sigma$ can be expressed as $f_i(y)=a_{i12}y+b_{i1},i=1,2$, and the crossing set is
\[
  \Sigma_c:=\{(0,y)\in\Sigma:f_1(y)f_2(y)>0\}.
\]
At these points, the two vector fields cross $\Sigma$ in the same direction, and their solutions are concatenated according to Filippov's convention~\cite{Filippov1988}.

\begin{definition}\label{def:crossing-cycle}
A \emph{crossing periodic orbit} of system~\eqref{eq:raw-system} is a nonconstant periodic trajectory $\Gamma=\{(x(t),y(t)) : t\in\mathbb{R}\}$ such that $\Gamma$ intersects both $\Omega_-$ and $\Omega_+$, $\Gamma\cap\Sigma$ is finite, and $\Gamma\cap\Sigma\subset\Sigma_c$. A \emph{crossing limit cycle} is a crossing periodic orbit that is isolated in the set of all crossing periodic orbits of~\eqref{eq:raw-system}.
\end{definition}

Freire, Ponce, and Torres~\cite{FreirePonceTorres2012Canonical}
proved that $a_{112}a_{212}>0$ is necessary for the existence
of a crossing periodic orbit. Under this condition, they
introduced the piecewise affine transformation
\[
(X,Y)=
\begin{cases}
\bigl(x,\ a_{122}x-a_{112}y-b_{11}\bigr), & x\le0,\\[4pt]
\left(
\dfrac{a_{112}}{a_{212}}x,\
\dfrac{a_{112}a_{222}}{a_{212}}x-a_{112}y-b_{11}
\right), & x>0,
\end{cases}
\]
which transforms~\eqref{eq:raw-system} into the Li\'enard
canonical form
\begin{equation}\label{eq:lienard-canonical}
\begin{aligned}
X<0:\quad&
\begin{cases}
\dot X=T_{L}X-Y,\\
\dot Y=D_{L}X-a_L,
\end{cases}
&
X>0:\quad&
\begin{cases}
\dot X=T_{R}X-Y+\delta,\\
\dot Y=D_{R}X-a_R.
\end{cases}
\end{aligned}
\end{equation}
Here
\[
T_L=\operatorname{tr}\mathbf A_1,\qquad
T_R=\operatorname{tr}\mathbf A_2,\qquad
D_L=\det\mathbf A_1,\qquad
D_R=\det\mathbf A_2,
\]
and
\[
a_L=a_{112}b_{12}-a_{122}b_{11},\qquad
a_R=\frac{a_{112}}{a_{212}}
     (a_{212}b_{22}-a_{222}b_{21}),\qquad
\delta=\frac{a_{112}}{a_{212}}b_{21}-b_{11}.
\]
In particular, systems~\eqref{eq:raw-system} and~\eqref{eq:lienard-canonical} have the same number of crossing limit cycles.

For continuous systems, the Lum--Chua conjecture~\cite{LumChua1991} asserts that there is at most one limit cycle, which was proved by Freire, Ponce, Rodrigo, and Torres~\cite{FreireEtAl1998}, and subsequently by Carmona, Fern\'andez-S\'anchez, and Novaes~\cite{CarmonaEtAl2021LumChua}. 

For discontinuous systems, Han and Zhang~\cite{HanZhang2010} constructed two limit cycles and conjectured that two is the maximum. Huan and Yang~\cite{HuanYang2012} found a numerical three-cycle example, whose existence was rigorously proved by Llibre and Ponce~\cite{LlibrePonce2012}. Further three-cycle constructions and examples were studied in~\cite{FreirePonceTorres2014Matching,FreirePonceTorres2014Mechanism,CarmonaEtAl2025Examples}, while upper-bound questions under additional structural hypotheses were investigated in~\cite{LlibreZhang2018Center,Huan2021,CarmonaEtAl2023Sewing}.
For the unrestricted two-zone class, Carmona, Fern\'andez-S\'anchez, and Novaes~\cite{CarmonaEtAl2023Bound} established a uniform upper bound of eight crossing limit cycles, which is not sharp. 

Our main result is the following.

\begin{theorem}\label{thm:main}
Any planar discontinuous piecewise linear system~\eqref{eq:raw-system} with two zones separated by a straight line has at most three crossing limit cycles.
\end{theorem}
Together with the three-cycle examples in~\cite{LlibrePonce2012,FreirePonceTorres2014Matching}, Theorem~\ref{thm:main} shows that three is the sharp uniform upper bound.

The remainder of this article is organized as follows. Section~\ref{sec:normalization} introduces the Poincar\'e half-return maps, establishes their basic analytical properties, and classifies the parameter regimes of the Li\'enard canonical form. Section~\ref{sec:matching} derives the matching equation whose isolated solutions correspond to crossing limit cycles, and introduces an auxiliary function $Y$ to bound the number of its solutions. Finally, Section~\ref{sec:regimes} completes the proof of Theorem~\ref{thm:main} across each parameter regime.

\section{Poincar\'e half-return maps and preliminary estimates}\label{sec:normalization}
\subsection{Two comparison half-systems}\label{subsec:model}
On $\Sigma=\{X=0\}$, the normal components of~\eqref{eq:lienard-canonical} are $-Y$ from the left and $-Y+\delta$ from the right, so the crossing set splits into the two rays
\[
  \Sigma^+:=\{(0,Y) : Y>\max(0,\delta)\},\qquad
  \Sigma^-:=\{(0,Y) : Y<\min(0,\delta)\}.
\]
Since $\mathrm{d}X/\mathrm{d}t<0$ on $\Sigma^+$ and $\mathrm{d}X/\mathrm{d}t>0$ on $\Sigma^-$, every crossing periodic orbit rotates counterclockwise (see Fig.~\ref{fig:half-maps}): it enters the left zone $X<0$ from a point $(0,u)\in\Sigma^+$, reaches $\Sigma^-$ at $(0,v)$, and returns through the right zone $X>0$ from $(0,v)$ to $(0,u)$.

Following~\cite[Section~6]{CarmonaEtAl2023Properties}, we set $(\xi_1(t),\eta_1(t)):=(X(t),Y(t))$ for the forward flow in $X<0$, and apply the reflection, shift, and time reversal
\begin{equation}\label{eq:right-to-left}
  \xi_2(t):=-X(-t),\qquad \eta_2(t):=Y(-t)-\delta
\end{equation}
to the backward flow in $X>0$. Then both the forward left excursion (from $(0,u)$ to $(0,v)$) and the backward right excursion (from $(0,u-\delta)$ to $(0,v-\delta)$) satisfy the same left half-plane system
\begin{equation}\label{eq:normal-pair}
  \xi_i' = T_i\xi_i-\eta_i,\qquad
  \eta_i' = D_i\xi_i-a_i,\qquad
  \xi_i<0,\quad i=1,2,
\end{equation}
with parameters
\[
  (T_1,D_1,a_1):=(T_L,D_L,a_L)\quad\text{and}\quad
  (T_2,D_2,a_2):=(-T_R,D_R,-a_R).
\]
In matrix form,
\[
  z_i'=\widehat{\mathbf A}_i z_i-\binom{0}{a_i},\qquad
  z_i:=\binom{\xi_i}{\eta_i},\qquad
  \widehat{\mathbf A}_i:=\begin{pmatrix}T_i&-1\\D_i&0\end{pmatrix},\qquad i=1,2.
\]

\begin{figure}[H]
  \centering
  \includegraphics[width=\linewidth]{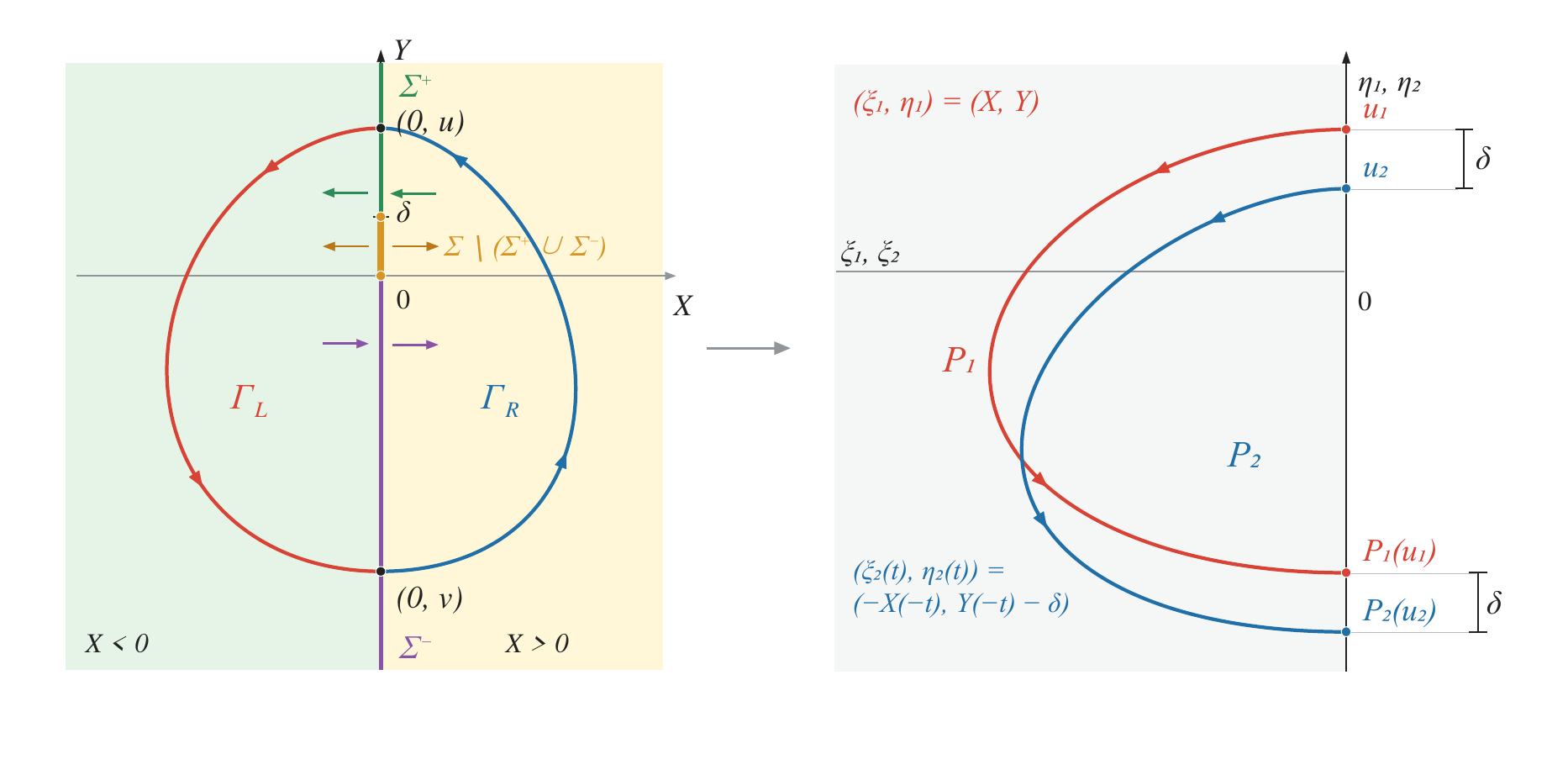}
  \caption{A crossing periodic orbit and its two comparison half-return arcs for $\delta>0$.}\label{fig:half-maps}
\end{figure}

\subsection{Return-time parametrization of the half-return map}\label{subsec:half-map-def}

Dropping the subscript $i$ from $z_i,\widehat{\mathbf A}_i$ and the parameters, we consider a single normalized subsystem on the open left half-plane:
\begin{equation}\label{eq:sub-normal}
  \left\{
  \begin{aligned}
    \xi'(t) &= T\xi(t)-\eta(t),\\
    \eta'(t) &= D\xi(t)-a,
  \end{aligned}
  \right.
  \qquad \xi(t)<0,
\end{equation}
For $u>0$, let $(\xi(t;u),\eta(t;u))$ be the solution of~\eqref{eq:sub-normal} starting at $(\xi(0),\eta(0))=(0,u)$. Whenever there exists a first return time $\tau=\tau(u)>0$ such that
\[
  \xi(t;u)<0 \quad \text{for all } t\in(0,\tau), \qquad \xi(\tau;u)=0, \qquad v:=\eta(\tau;u)<0,
\]
we define the \emph{Poincar\'e half-return map} by $P(u):=v$, and denote by $I\subset(0,\infty)$ its domain of definition.

Eliminating $\eta(t)=T\xi(t)-\xi'(t)$ from~\eqref{eq:sub-normal} shows that $\xi(t)$ satisfies the second-order initial value problem
\begin{equation}\label{eq:xi-ode}
  \xi''(t)-T\xi'(t)+D\xi(t)=a, \qquad \xi(0)=0,\quad \xi'(0)=-u.
\end{equation}
Let $S(t)$ be the solution of the homogeneous initial value problem
\begin{equation}\label{eq:kernel}
  S''(t)-TS'(t)+DS(t)=0, \qquad S(0)=0,\quad S'(0)=1.
\end{equation}
Setting $\Delta:=T^2-4D$, the explicit form of $S(t)$ is
\begin{equation}\label{eq:S-explicit}
  S(t)=
  \begin{cases}
    \mathrm{e}^{Tt/2}\,\dfrac{\sinh(\sqrt{\Delta}\,t/2)}{\sqrt{\Delta}/2}, & \Delta>0,\\[8pt]
    t\,\mathrm{e}^{Tt/2}, & \Delta=0,\\[6pt]
    \mathrm{e}^{Tt/2}\,\dfrac{\sin(\sqrt{-\Delta} t/2)}{\sqrt{-\Delta}/2}, & \Delta<0.
  \end{cases}
\end{equation}
In terms of $S(t)$, we define the auxiliary functions
\begin{equation}\label{eq:BN-def}
  B(t):=\int_0^t S(\theta)\,\mathrm{d}\theta, \qquad
  N(t):=S(t)^2-B(t)S'(t).
\end{equation}
Integrating~\eqref{eq:kernel} over $[0,t]$ and substituting~\eqref{eq:kernel} into the derivative of~\eqref{eq:BN-def} yield the identities
\begin{equation}\label{eq:BN-identities}
  S'(t)-TS(t)+DB(t)=1, \qquad
  N'(t)-TN(t)=S(t).
\end{equation}

By variation of parameters, the solution of~\eqref{eq:xi-ode} is given by
\begin{equation}\label{eq:sub-sol}
  \xi(t)=-uS(t)+aB(t), \qquad
  \xi'(t)=-uS'(t)+aS(t).
\end{equation}

the return relations $\xi(\tau)=0$ and $v=\eta(\tau)=-\xi'(\tau)$ take the form
\begin{equation}\label{eq:uv-linear}
  \left\{
  \begin{aligned}
    uS(\tau) &= aB(\tau),\\
    v-uS'(\tau) &= -aS(\tau).
  \end{aligned}
  \right.
\end{equation}
Whenever $S(\tau)\ne0$, solving~\eqref{eq:uv-linear} and substituting~\eqref{eq:BN-def} gives the return-time parametrization
\begin{equation}\label{eq:UV-param}
  u=U(\tau):=\frac{aB(\tau)}{S(\tau)}, \qquad
  v=V(\tau):=-\frac{aN(\tau)}{S(\tau)}.
\end{equation}

The domain and range of the half-return map follow from~\cite[Section~2, pp.~324--325, and Corollary~21]{CarmonaFernandez2021}. We state them for the strict return domain $u>0>P(u)$ used here.

\begin{proposition}[\cite{CarmonaFernandez2021}]\label{prop:half-map-domain}
The half-return map of~\eqref{eq:sub-normal} has a nonempty domain if and only if $a>0$, or $a\le0$ and $D>T^2/4$. Its domain $I$ and range $P(I)$ are as follows, with $\omega:=\sqrt{D-T^2/4}$ whenever $D>T^2/4$.
\begin{enumerate}
  \item If $a>0$ and $D\le T^2/4$, set
  \[
    \rho:=T/2+\sqrt{T^2/4-D},\qquad
    \rho_-:=-T/2+\sqrt{T^2/4-D}.
  \]
  Then $I=(0,U_*)$ and $P(I)=(V_*,0)$, where
  \[
    U_*=
    \begin{cases}
      a/\rho,&\rho>0,\\
      \infty,&\rho\le0,
    \end{cases}
    \qquad
    V_*=
    \begin{cases}
      -a/\rho_-,&\rho_->0,\\
      -\infty,&\rho_-\le0.
    \end{cases}
  \]
  \item If $a>0$ and $D>T^2/4$, then $I=(0,\infty)$ and $P(I)=(-\infty,0)$.
  \item If $a=0$ and $D>T^2/4$, then $I=(0,\infty)$, $P(I)=(-\infty,0)$, and
  \begin{equation}\label{eq:P-a0}
    P(u)=-\mathrm{e}^{\pi T/(2\omega)}u.
  \end{equation}
  \item If $a<0$ and $D>T^2/4$, then
  \begin{enumerate}
    \item for $T>0$, $I=(0,\infty)$ and $P(I)=(-\infty,v_b)$, where $v_b:=\lim_{u\downarrow0}P(u)<0$;
    \item for $T=0$, $I=(0,\infty)$, $P(I)=(-\infty,0)$, and $P(u)=-u$;
    \item for $T<0$, $I=(u_c,\infty)$ and $P(I)=(-\infty,0)$, where $u_c>0$ and $\lim_{u\downarrow u_c}P(u)=0$.
  \end{enumerate}
\end{enumerate}
\end{proposition}
By continuity, $P$ extends to the left endpoint of $I$ via $P(0):=0$ when $a\ge0$ or $a<0=T$, $P(0):=v_b$ when $a<0<T$, and $P(u_c):=0$ when $a<0$ and $T<0$; these are limiting endpoint values and do not enlarge the domain $I$, whereas interior points $u\in I$ correspond to transversal crossings with $u>0>P(u)$.
\subsection{Derivatives and concavity of the half-return map}\label{subsec:derivatives}

The rational formulas for $P'$ and $P''$, together with the sign relations for $u+P(u)$ and $P''(u)$, were established by Carmona, Fern\'andez-S\'anchez, Garc\'{\i}a-Medina, and Novaes~\cite[Propositions~4.1, 5.1, and~5.2]{CarmonaEtAl2023Properties}. We recall these results and give a derivation based on Liouville's formula and an energy identity.

\begin{lemma}\label{lem:first-derivative-energy}
The half-return map $P$ is real analytic on $I$. For $u\in I$, let $v=P(u)<0$ and $\tau=\tau(u)$. Then
\begin{equation}\label{eq:Pprime-liouville}
  P'(u)=\frac{u}{v}\mathrm{e}^{T\tau}<0,
\end{equation}
and
\begin{equation}\label{eq:energy}
  v^2-u^2=2T\int_0^\tau \xi'(t)^2\,\mathrm{d}t.
\end{equation}
In particular,
\begin{equation}\label{eq:sgnu+v}
  \operatorname{sgn}(u+P(u))=-\operatorname{sgn}T,\qquad u\in I.
\end{equation}
\end{lemma}

\begin{proof}
Since $\partial_t\xi(\tau(u);u)=-v>0$, the analytic implicit-function theorem gives the analyticity of $\tau$ and $P$. Set $q=\partial_u z$. Both $z'$ and $q$ satisfy the variational equation $Y'=\widehat{\mathbf A}Y$. Differentiating $z(\tau(u);u)=(0,P(u))^{\top}$ gives $q(\tau)+z'(\tau)\tau'=(0,P')^{\top}$. Liouville's formula therefore yields
\[
  -vP'=\det(z'(\tau),q(\tau))
  =\mathrm{e}^{T\tau}\det(z'(0),q(0))=-u\mathrm{e}^{T\tau},
\]
which proves~\eqref{eq:Pprime-liouville}. Multiplying~\eqref{eq:xi-ode} by $2\xi'$ and integrating over $[0,\tau]$ gives~\eqref{eq:energy}. Its integral is positive because $\xi'(0)=-u\ne0$; hence $v^2-u^2=(v-u)(v+u)$ and $v-u<0$ imply~\eqref{eq:sgnu+v}.
\end{proof}

Following~\cite[equation~(2.3)]{CarmonaEtAl2023Properties}, write
\begin{equation}\label{eq:W-def}
  W(y):=Dy^2-aTy+a^2.
\end{equation}
To recover the rational derivative formulas, observe that at a boundary point $(0,y)$,
\[
  \det(z',z'')=\det\begin{pmatrix}-y&a-Ty\\-a&-Dy\end{pmatrix}=W(y).
\]
Both $z'$ and $z''$ satisfy $Y'=\widehat{\mathbf A}Y$, so Liouville's formula gives $W(v)=\mathrm{e}^{T\tau}W(u)$.

\begin{lemma}\label{lem:derivatives}
For $a\ne0$ and $u\in I$, let $v=P(u)$ and $\tau=\tau(u)$. Then
\begin{equation}\label{eq:W-pos}
  W(u)=\frac{a(u-v)}{S(\tau)}>0,\qquad W(v)=\mathrm{e}^{T\tau}W(u)>0,
\end{equation}
and
\begin{equation}\label{eq:P-derivatives}
  P'(u)=\frac{uW(v)}{vW(u)},\qquad
  P''(u)=\frac{a^2(v^2-u^2)W(v)}{v^3W(u)^2}.
\end{equation}
For all parameter values with nonempty return domain,
\begin{equation}\label{eq:concavity}
  \operatorname{sgn}P''(u)=-\operatorname{sgn}(a^2T),\qquad u\in I.
\end{equation}
\end{lemma}

\begin{proof}
If $a>0$, then $S(\tau)>0$: otherwise, the first positive zero $p\le\tau$ of $S$ would satisfy $\xi(p)=aB(p)>0$, contrary to the definition of the first return. Equations~\eqref{eq:uv-linear} and~\eqref{eq:BN-identities} therefore give
\[
  W(u)=\frac{au(1-S'(\tau))+a^2S(\tau)}{S(\tau)}
      =\frac{a(u-v)}{S(\tau)}>0.
\]
If $a<0$, Proposition~\ref{prop:half-map-domain} gives $D>T^2/4$, so $W>0$ on $\mathbb{R}$; the same equations give $W(u)S(\tau)=a(u-v)$. Together with $W(v)=\mathrm{e}^{T\tau}W(u)$, this proves~\eqref{eq:W-pos} and the rational formula for $P'$.

Differentiation of $\log(-P')$ yields
\[
  P''
  =\frac{uP'}{W(u)}\left[
    \left.\frac{\mathrm{d}}{\mathrm{d}y}\frac{W(y)}y\right|_{y=v}
    -\left.\frac{\mathrm{d}}{\mathrm{d}y}\frac{W(y)}y\right|_{y=u}
  \right]
  =\frac{a^2(v^2-u^2)W(v)}{v^3W(u)^2}.
\]
Since $v<0$ and $W(u),W(v)>0$, equation~\eqref{eq:energy} gives~\eqref{eq:concavity}. For $a=0$, the same conclusion follows from~\eqref{eq:P-a0}.
\end{proof}

\subsection{Parameter classification}\label{subsec:classification}
Assume $I_1,I_2\ne\varnothing$, since otherwise no crossing periodic orbit exists. The case $a_1^2a_2^2T_1T_2\le0$ will be treated by concavity in Subsection~\ref{sec:concavity}. When $a_{1}^2a_{2}^2T_{1}T_{2}>0$, we may assume $T_1,T_2>0$ by simultaneous time reversal; see~\eqref{eq:time-reverse-map}. Set
\begin{equation}\label{eq:nu-sigma-lambda}
  \nu_i:=\frac{D_i}{T_i^2},\qquad i=1,2.
\end{equation}
These parameters fall into four disjoint cases:
\begin{enumerate}
  \item $a_1a_2<0$ (Subsection~\ref{sec:mixed});
  \item $a_1,a_2>0$ and $\nu_1\ne\nu_2$ (Subsection~\ref{sec:positive});
  \item $a_1,a_2<0$ and $\nu_1\ne\nu_2$ (Subsection~\ref{sec:negative});
  \item $a_1a_2>0$ and $\nu_1=\nu_2$ (Subsection~\ref{sec:same-shape}).
\end{enumerate}
The proofs are given in Section~\ref{sec:regimes}. Case~(4) uses scaling; cases~(1)--(3) use the auxiliary function $Y$ and the counting criterion developed in Subsection~\ref{sec:universal}. 

\section{The matching equation and a uniform counting criterion}\label{sec:matching}

\subsection{The matching equation and correspondence with limit cycles}\label{sec:geometry}

Since $P_1$ and $P_2$ are strictly decreasing on $I_1$ and $I_2$, the full Poincaré return map $\Pi(u):=\delta+P_2^{-1}(P_1(u)-\delta)$ is strictly increasing and real analytic wherever defined, and hence admits no periodic points of minimal period greater than one~\cite[Section~3]{FreirePonceTorres2012Canonical}. Consequently, every crossing periodic orbit of~\eqref{eq:lienard-canonical} meets $\Sigma^+$ at a unique entry point $(0,u_1)$ and $\Sigma^-$ at a unique exit point $(0,P_1(u_1))$. By Subsection~\ref{subsec:model}, such an orbit closes if and only if $(u_1,u_2)\in I_1\times I_2$ satisfies the matching system
\begin{equation}\label{eq:matching}
  u_1-u_2=\delta,\qquad P_1(u_1)-P_2(u_2)=\delta,\qquad (u_1,u_2)\in I_1\times I_2.
\end{equation}
Eliminating $u_2=u_1-\delta$ reduces the problem of counting crossing limit cycles of~\eqref{eq:raw-system} to counting the isolated zeros of the displacement function
\begin{equation}\label{eq:delta}
  \Delta_\delta(u):=P_2(u-\delta)+\delta-P_1(u),\qquad u\in I_\delta:=I_1\cap(\delta+I_2).
\end{equation}

For each half-return map $P_i$, define
\begin{equation}\label{eq:chord}
  \ell_i(u):=\frac{u-P_i(u)}{2},\qquad
  m_i(u):=\frac{u+P_i(u)}{2},\qquad u\in I_i.
\end{equation}
Since $\ell_i'=(1-P_i')/2>0$, $\ell_i:I_i\to J_i:=\ell_i(I_i)$ is a real-analytic diffeomorphism. Let $u_i:=\ell_i^{-1}$ and $M_i(\ell):=m_i(u_i(\ell))$ for $\ell\in J_i$. Then
\[
  u_i(\ell)=M_i(\ell)+\ell,\qquad
  P_i(u_i(\ell))=M_i(\ell)-\ell.
\]
In these coordinates,~\eqref{eq:matching} is equivalent to
\[
  \ell_1=\ell_2=\ell,\qquad M_1(\ell_1)-M_2(\ell_2)=\delta,\qquad
  (\ell_1,\ell_2)\in J_1\times J_2.
\]
Equivalently,
\begin{equation}\label{eq:chord-matching}
  \mathcal{B}(\ell):=M_1(\ell)-M_2(\ell)=\delta,\qquad \ell\in J_1\cap J_2.
\end{equation}
\begin{figure}[htbp]
  \centering
  \includegraphics[width=\linewidth]{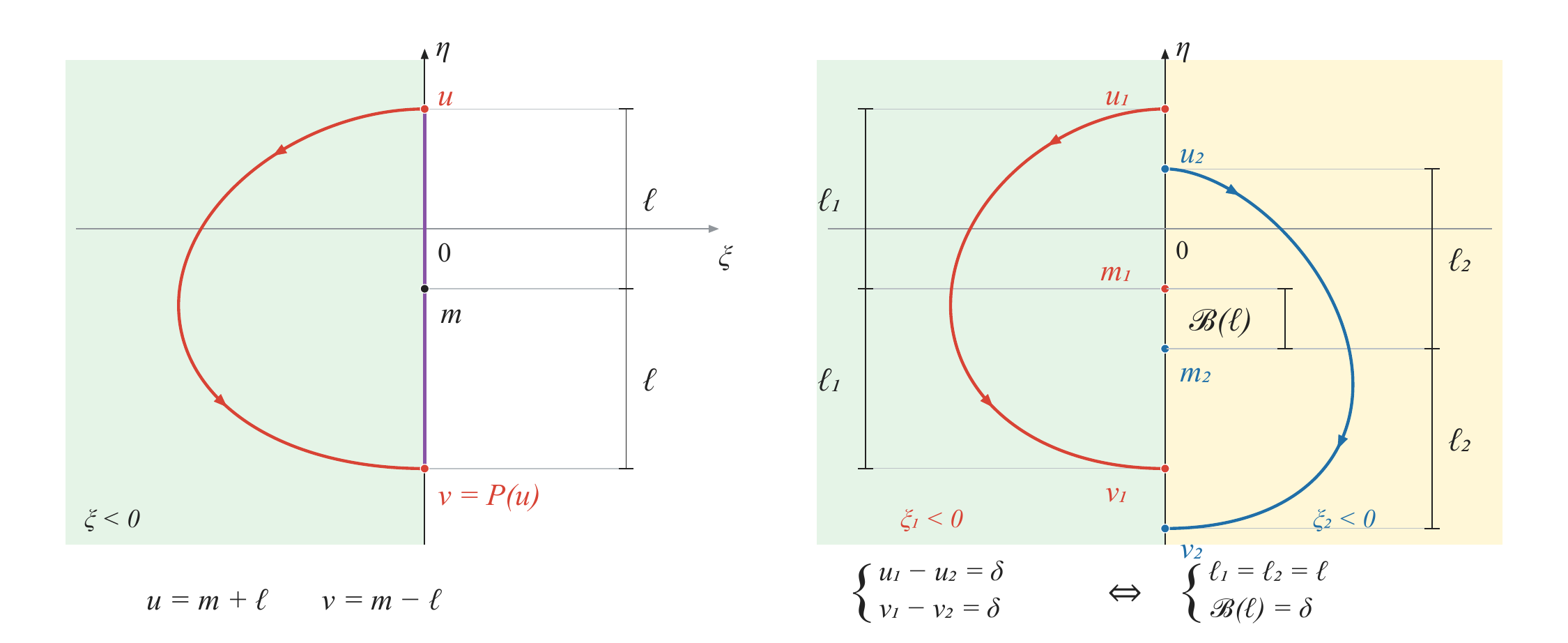}
  \caption{Boundary endpoints and their half-length and midpoint coordinates. The second half-plane is reflected for display.}\label{fig:half-length-matching}
\end{figure}

The change of variable $\ell=\ell_1(u)$ preserves isolation of solutions. Hence the crossing limit cycles of~\eqref{eq:raw-system} correspond bijectively to the isolated solutions of~\eqref{eq:chord-matching}.
\subsection{A counting criterion for the matching equation}\label{sec:universal}
Consider the case $a_1^2a_2^2T_1T_2>0$ identified in Subsection~\ref{subsec:classification}. The transformation $(t,\xi,\eta)\mapsto(-t,\xi,-\eta)$ changes $T$ to $-T$, leaves $D,a$ unchanged, and gives
\begin{equation}\label{eq:time-reverse-map}
  P_{-T,D,a}(u)=-P_{T,D,a}^{-1}(-u),\qquad u\in-P_{T,D,a}(I).
\end{equation}
Applied to both half-maps, $(u_1,u_2)\mapsto(-P_1(u_1),-P_2(u_2))$ is a homeomorphism between the matching solution sets for $(T_1,T_2,\delta)$ and $(-T_1,-T_2,-\delta)$. Thus the number of isolated solutions is unchanged, and we may assume
\begin{equation}\label{eq:positive-T}
  T_1,T_2>0 \qquad\text{and}\qquad a_1a_2\ne0.
\end{equation}

For a half-map, omit the subscript $i$. Equations~\eqref{eq:energy} and~\eqref{eq:concavity} give $P(u)<-u$ and $P''(u)<0$. Set
\[
  s=-P'(u),\qquad K=-P'(I).
\]
Since $\mathrm{d}s/\mathrm{d}u=-P''(u)>0$, the inverse function theorem gives an analytic inverse $u=u(s)$ on the open interval $K$. Evaluating $\ell,m$ at $u(s)$, define
\begin{equation}\label{eq:s-h-K}
  h(s):=-\frac{m}{\ell}\in(0,1),\qquad
  \kappa(s):=-\ell P''(u(s))>0.
\end{equation}
Then $u=\ell(1-h)$ and $P(u)=-\ell(1+h)$. For the two half-maps, let $K_i=-P_i'(I_i)$ and set
\begin{equation}\label{eq:Y-comparison}
  Y(s):=\kappa_2(s)-\kappa_1(s),\qquad s\in K_1\cap K_2.
\end{equation}
For $\ell\in J_i$, set $s_i(\ell):=-P_i'(u_i(\ell))$. The chain rule gives
\begin{equation}\label{eq:Mprime}
  M_i'(\ell)=\frac{1-s_i(\ell)}{1+s_i(\ell)},\qquad
  M_i''(\ell)=\frac{4P_i''(u_i(\ell))}{(1+s_i(\ell))^3}.
\end{equation}
The zeros of $Y$ provide the following bound for the solutions of~\eqref{eq:chord-matching}.

\begin{theorem}\label{thm:uniform-count}
Suppose that $T_i>0$ and $a_i\ne0$ for $i=1,2$. If $Y(s)=\kappa_2(s)-\kappa_1(s)$ has at most one zero on $K_1\cap K_2$, then for every $\delta\in\mathbb{R}$, the matching equation $\mathcal{B}(\ell)=\delta$ has at most three distinct solutions on $J_1\cap J_2$.
\end{theorem}
\begin{proof}
If $K_1\cap K_2=\varnothing$, then $\BB'$ never vanishes by~\eqref{eq:Mprime}, and Rolle's theorem gives at most one solution. Otherwise, set
\[
  \mathcal L(s):=\log\ell_1(s)-\log\ell_2(s),\qquad s\in K_1\cap K_2.
\]
The chain rule gives $(\log\ell_i)'(s)=(1+s)/(2\kappa_i(s))$, and therefore
\begin{equation}\label{eq:log-count}
  \mathcal L'(s)=\frac{1+s}{2\kappa_1(s)\kappa_2(s)}Y(s).
\end{equation}
Thus $\mathcal L$ has at most two zeros. By~\eqref{eq:Mprime},
\[
  \BB'(\ell)=0
  \quad\Longrightarrow\quad
  s_1(\ell)=s_2(\ell)=:s,\qquad
  \ell_1(s)=\ell_2(s)=\ell,\qquad \mathcal L(s)=0.
\]
Strict monotonicity of $s_1(\ell)$ makes this correspondence injective. Hence $\BB'$ has at most two zeros, and Rolle's theorem gives the bound.
\end{proof}

\section{Proof by parameter classes}\label{sec:regimes}
We consider the parameter classes in Subsection~\ref{subsec:classification}.

\subsection{Cases settled by concavity}\label{sec:concavity}

\begin{proposition}\label{prop:concavity-cases}
If $a_{1}^2a_{2}^2T_{1}T_{2}\leq 0$, then~\eqref{eq:raw-system} has at most two crossing limit cycles. In particular, if $a_1T_1=a_2T_2=0$, then~\eqref{eq:raw-system} has at most one crossing limit cycle.
\end{proposition}

\begin{proof}
If $a_1T_1=a_2T_2=0$, then $\Delta_\delta''=0$, so $\Delta_\delta$ has at most one isolated zero. Otherwise,~\eqref{eq:concavity} gives $\Delta_\delta''\ne0$ on $I_\delta$, and two applications of Rolle's theorem give at most two zeros. The cycle bounds follow from Subsection~\ref{sec:geometry}.
\end{proof}

\subsection{The case \texorpdfstring{$a_1a_2<0$}{a1 a2 < 0}}\label{sec:mixed}
Without loss of generality, assume that \( a_1>0>a_2\).

\begin{lemma}\label{lem:mixed-heights}
  For every $s\in K_1\cap K_2$,
  \begin{equation}\label{eq:mixed-heights}
    0<h_1(s)<\frac{s-1}{s+1}<h_2(s)<1.
  \end{equation}
\end{lemma}

\begin{proof}
For $u=u_1(s)$, strict concavity and $P_1(0)=0$ imply
\[
  s=-P_1'(u)>-\frac{P_1(u)}u=\frac{1+h_1}{1-h_1}>1.
\]
This gives $h_1<(s-1)/(s+1)$. For the second map, substituting $u_2=\ell_2(1-h_2)$ and $P_2(u_2)=-\ell_2(1+h_2)$ into~\eqref{eq:P-derivatives} gives
\[
  \frac{(s+1)h_2-(s-1)}{2(1-h_2)}
  =\frac{2h_2a_2^2-a_2T_2\ell_2(1-h_2^2)}
  {(1-h_2^2)W_2(u_2)}>0,
\]
since $a_2<0$. Hence $h_2>(s-1)/(s+1)$.
\end{proof}

\begin{lemma}\label{lem:downcrossing}
Let $f$ be continuous on an open interval $I$ and differentiable at its zeros. If $f'$ has a constant nonzero sign at the zeros of $f$, then $f$ has at most one zero in $I$.
\end{lemma}

\begin{proof}
Replacing $f$ by $-f$ if necessary, assume $f'<0$ at every zero. If $\alpha<\beta$ were zeros, there would exist $\alpha<x<y<\beta$ with $f(x)<0<f(y)$. Let $c=\min\{t\in[x,y]:f(t)=0\}$. Then $f<0$ on $[x,c)$, so 
\[
  f'(c)=\lim_{t\uparrow c}\frac{f(t)}{t-c}\ge0,
\]
a contradiction.
\end{proof}

\begin{proposition}\label{prop:mixed}
  If $T_1,T_2>0$ and $a_1a_2<0$, then
  \[
    M_1(\ell)-M_2(\ell)=\delta,\qquad \ell\in J_1\cap J_2,
  \]
  has at most three distinct solutions.
\end{proposition}

\begin{proof}
By Lemma~\ref{lem:downcrossing} and Theorem~\ref{thm:uniform-count}, it suffices to prove that $Y'(s)<0$ whenever $Y(s)=0$.

Since $\kappa=-\ell P''$, $\ell_u=(1+s)/2$ and $s_u=\kappa/\ell$, the chain rule gives
\begin{equation}\label{eq:kappa-s}
  \kappa_s=\frac{1+s}{2}-\frac{\ell^2P'''(u)}{\kappa}.
\end{equation}
At a zero of $Y$, put $\kappa=\kappa_1(s)=\kappa_2(s)>0$. Then
\[
  Y'(s)=\frac{\ell_1(s)^2P_1'''(u_1(s))-\ell_2(s)^2P_2'''(u_2(s))}{\kappa}.
\]
For any half-return map with $T>0$ and $a\ne0$, Appendix~\ref{app:third-derivative} expresses the third derivative as
\begin{equation}\label{eq:J}
  \begin{aligned}
    \ell^2P'''(u)=\mathcal J(h,s,\kappa)
    &:=-\frac{(3h+1)\kappa^2}{2hs}\\
    &\quad+\frac{\kappa\bigl[(3h^2+1)(1-s)+4h(s+1)\bigr]}
    {2h(1-h^2)}.
  \end{aligned}
\end{equation}
Substitution gives
\begin{equation}\label{eq:contact-derivative}
  Y'(s)=\frac{\mathcal J(h_1(s),s,\kappa)-\mathcal J(h_2(s),s,\kappa)}{\kappa}.
\end{equation}
For fixed $s,\kappa>0$,
\begin{equation}\label{eq:J-h-large}
  \partial_h\mathcal J(h,s,\kappa)
  =\frac{\kappa}{2h^2s}\left(
  \kappa+\frac{s\bigl[(s-1)(1-h)^3(3h+1)+16h^3\bigr]}
  {(1-h^2)^2}\right),
\end{equation}
which is positive when $s\ge1,0<h<1$. Thus lemma~\ref{lem:mixed-heights} implies $Y'<0$, the conclusion follows. 
\end{proof}

\subsection{The case \texorpdfstring{$a_1,a_2>0,\nu_1\ne\nu_2$}{a1, a2 > 0,nu1!=nu2}}\label{sec:positive}
Strict concavity and $P_i(0)=0$ give
\begin{equation}\label{eq:positive-height}
  s>1,\qquad 0<h_i(s)<\frac{s-1}{s+1},\qquad \kappa_i(s)>0,
\end{equation}
as in the first part of Lemma~\ref{lem:mixed-heights}. The following lemma will be used to prove that $Y$ has at most one zero.

\begin{lemma}\label{lem:axis-contact}
  Let $f,g\in C^1(I)$ where \(I\) is an open interval. Suppose they never vanish simultaneously and satisfy
  \[
\begin{gathered}
    \sgn f'(s)=\sgn g(s)\quad\text{when }f(s)=0,\\
    \sgn g'(s)=\sgn f(s)\quad\text{when }g(s)=0.
  \end{gathered}
\]
  Then $g$ has at most one zero.
\end{lemma}

\begin{proof}
Since $f$ and $g$ have no common zero, the hypotheses give
\[
  (fg)'(s)=f'(s)g(s)+f(s)g'(s)>0
  \qquad\text{whenever }f(s)g(s)=0.
\]
Lemma~\ref{lem:downcrossing} implies that $fg$, and hence $g$, has at most one zero.
\end{proof}

We shall apply Lemma~\ref{lem:axis-contact} to $h_1-h_2$ and $Y$. Define
\begin{equation}\label{eq:shape-recovery}
  \mathcal V(h,s,\kappa):=
  \frac{\kappa(1-h^2)^2\bigl[2hs(s+1)-\kappa(1-h^2)\bigr]}
  {4h^2\bigl[s\bigl((s+1)h-(s-1)\bigr)-\kappa(1-h^2)\bigr]^2}
\end{equation}
for $0<h<1$ and $s,\kappa>0$.

\begin{lemma}\label{lem:shape-recovery}
Fix a system~\eqref{eq:sub-normal} with $T>0$, $a\ne0$, and nonempty half-return domain $I$. Let $K=-P'(I)$ be its corresponding parameter interval. Then $\mathcal V(h(s),s,\kappa(s))$ is defined for every $s\in K$ and
\[
  \mathcal V(h(s),s,\kappa(s))\equiv\nu,\qquad s\in K.
\]
\end{lemma}

\begin{proof}
Fix $s\in K$, set $u=u(s),h=h(s),\kappa=\kappa(s)$, and write
\[
  W(\ell\xi)=A\xi^2+B\xi+C,
  \qquad
  (A,B,C)=(D\ell^2,-aT\ell,a^2).
\]
Thus $B\ne0$ and $AC/B^2=D/T^2=\nu$. With $h=h(s)$ and $\kappa=\kappa(s)$, the endpoint relations and~\eqref{eq:P-derivatives} give
\[
  \begin{aligned}
    A(1-h)^2+B(1-h)+C&=W(u),\\
    A(1+h)^2-B(1+h)+C&=s\frac{1+h}{1-h}W(u),\\
    \frac{4hsC}{(1-h)(1+h)^2}&=\kappa W(u).
  \end{aligned}
\]
Solving yields
\begin{equation}\label{eq:reconstruction}
  \begin{aligned}
    A&=\big(\frac{1+s}{2(1-h)}-\frac{\kappa(1+h)}{4hs}\big)W(u),\\
    B&=\big(\frac{(s+1)h-(s-1)}{2(1-h)}-\frac{\kappa(1+h)}{2s}\big)W(u),\\
    C&=\big(\frac{\kappa(1-h)(1+h)^2}{4hs}\big)W(u).
  \end{aligned}
\end{equation}
Substitution into $AC/B^2$ gives~\eqref{eq:shape-recovery} and hence $\mathcal V(h(s),s,\kappa(s))\equiv \nu$. 
\end{proof}

\begin{proposition}\label{prop:positive}
  If $\nu_1=D_1/T_1^2\ne\nu_2=D_2/T_2^2$, the matching equation has at most three distinct solutions.
\end{proposition}
\begin{proof}
Set $f=h_1-h_2$. By Lemma~\ref{lem:shape-recovery},
\[
  \mathcal V(h_1(s),s,\kappa_1(s))-
  \mathcal V(h_2(s),s,\kappa_2(s))
  =\nu_1-\nu_2\ne0,\qquad s\in K_1\cap K_2.
\]
Hence $f$ and $Y$ have no common zero.

Differentiating $h=-m/\ell$ with respect to $s$ gives
\begin{equation}\label{eq:h-s}
  h_s=\frac{s-1-(s+1)h}{2\kappa}.
\end{equation}
At a zero of $f$, put $h=h_1=h_2$. Then
\[
  f'=\frac{s-1-(s+1)h}{2\kappa_1\kappa_2}Y,
\]
and~\eqref{eq:positive-height} gives $\sgn f'=\sgn Y$. At a zero of $Y$, equations~\eqref{eq:contact-derivative} and~\eqref{eq:J-h-large} give $\sgn Y'=\sgn f$. Lemma~\ref{lem:axis-contact} bounds the number of zeros of $Y$ by one, and Theorem~\ref{thm:uniform-count} gives the conclusion.
\end{proof}

\subsection{The case \texorpdfstring{$a_1,a_2<0,\nu_1\neq \nu_2$}{a1, a2 < 0,nu1!=nu2}}\label{sec:negative}
Proposition~\ref{prop:half-map-domain} gives
\[
  D_i>T_i^2/4,\qquad I_i=(0,\infty),\qquad i=1,2.
\]

\begin{lemma}\label{lem:negative-shape-order}
Assume $T_i>0$ and $a_i<0$ for $i=1,2$. For every $s\in K_1\cap K_2$,
\begin{equation}\label{eq:negative-contact-height}
  \kappa_1(s)=\kappa_2(s),\quad\nu_1>\nu_2
  \quad\Longrightarrow\quad h_1(s)<h_2(s).
\end{equation}
\end{lemma}

\begin{proof}
Fix $s\in K_1\cap K_2$ with $\kappa_1(s)=\kappa_2(s)=\kappa$ and put
\[
  r_i=\frac{1-h_i(s)}{1+h_i(s)}=\frac{u_i}{-P_i(u_i)},\qquad i=1,2.
\]
Equations~\eqref{eq:Pprime-liouville} and~\eqref{eq:energy} give $0<r_i<\min\{1,s\}$. Define
\[
  \mathcal N_{s,\kappa}(r):=s(1+r)(1-sr)-2\kappa r.
\]
Substitution into~\eqref{eq:reconstruction} gives
\[
  \frac{\mathcal N_{s,\kappa}(r_i)}{2sr_i(1+r_i)}
  =-\frac{a_{i}T_i\ell_i}{W_i(u_i)}>0.
\]
By Lemma~\ref{lem:shape-recovery}, $\nu_i=\mathcal R_{s,\kappa}(r_i)$, where
\begin{equation}\label{eq:ratio-shape}
  \mathcal R_{s,\kappa}(r):=\mathcal V\left(\frac{1-r}{1+r},s,\kappa\right)
  =\frac{2\kappa r^2\bigl[s(s+1)(1-r^2)-2\kappa r\bigr]}
  {(1-r)^2\mathcal N_{s,\kappa}(r)^2}.
\end{equation}

Put $\rho=\max\{r_1,r_2\}$. For $r\in[\min\{r_1,r_2\},\rho]$,
\begin{equation}\label{eq:ratio-prefix}
  \rho\mathcal N_{s,\kappa}(r)
  =r\mathcal N_{s,\kappa}(\rho)+s(\rho-r)+s^2\rho r(\rho-r)>0.
\end{equation}
Set $L=\mathcal N_{s,\kappa}(r)/(2r)>0$. Differentiating at fixed $s,\kappa$ and substituting $\kappa=s(1+r)(1-sr)/(2r)-L$ gives
\begin{align}
  \mathcal R_{s,\kappa}'(r)
  &=\frac{4\kappa r(1+r)\mathcal D(s,\kappa,r)}
  {(1-r)^3\mathcal N_{s,\kappa}(r)^3},\label{eq:ratio-derivative}\\
  \mathcal D(s,\kappa,r)
  &=s^2(1-r)(s-r)(1+sr^2)\nonumber\\
  &\quad+rs\bigl((3r-r^2)(s-r)+(1-r)^3\bigr)L+2r^2L^2.\label{eq:ratio-positive}
\end{align}
Every term in $\mathcal D$ is positive, so $\mathcal R_{s,\kappa}'>0$ between $r_1$ and $r_2$. Hence
\[
  \nu_1>\nu_2\quad\Longrightarrow\quad r_1>r_2
  \quad\Longrightarrow\quad h_1<h_2.
\]
\end{proof}

For $0<s\le1$, set
\begin{equation}\label{eq:barrier-functions}
  h_-(s):=\frac{1-s}{1+s},\qquad k_0(s):=\frac{(1-s)(1+s)^2}{2}.
\end{equation}
\begin{lemma}\label{lem:focus-barrier}
For a half-return map with $T>0$, $a<0$, and nonempty domain, every value $s=-P'(u)<1$ satisfies
\[
  h(s)>h_-(s),\qquad \kappa(s)>k_0(s).
\]
\end{lemma}

\begin{proof}
Set $Z(s):=\kappa(s)-k_0(s)$. By~\eqref{eq:Pprime-liouville},
\[
  s=\frac{u}{-P(u)}\ee^{T\tau(u)}
  >\frac{u}{-P(u)}=\frac{1-h}{1+h},
\]
which gives $h>h_-(s)$.

Fix $s_0\in-P'(I)$ with $s_0<1$. We first prove that $[s_0,1]\subset-P'(I)$ and $Z(1)>0$.
Put $p=\pi/\omega$, where $\omega=\sqrt{D-T^2/4}$. Since $S,B>0$ on $(0,p)$ and $B(p)>0$, equation~\eqref{eq:sub-sol} gives $\tau(u)>p$. For any $\varepsilon\in(0,p)$, $S(p+\varepsilon)<0$, so $\xi(p+\varepsilon;u)>0$ for sufficiently large $u$. Hence $p<\tau(u)<p+\varepsilon$ for such $u$, and therefore $\tau(u)\to p$. Equation~\eqref{eq:uv-linear} then yields $P(u)/u\to S'(p)=-\ee^{Tp/2}$, so
\begin{equation}\label{eq:negative-slope-limit}
  \lim_{u\to\infty}-P'(u)
  =\lim_{u\to\infty}\frac{u}{-P(u)}\ee^{T\tau(u)}
  =\ee^{Tp/2}>1.
\end{equation}
The continuity of $-P'$ gives $[s_0,1]\subset-P'(I)$. Thus $Z$ is defined on $[s_0,1]$, with $Z(1)=\kappa(1)>0$.

By~\eqref{eq:kappa-s},
\[
  \kappa_s=G(s,h,\kappa):=\frac{1+s}{2}
  -\frac{\mathcal J(h,s,\kappa)}{\kappa}.
\]
At any zero $s\in[s_0,1)$ of $Z$, we have $\kappa(s)=k_0(s)$. Writing $h=h(s)>h_-(s)$ and substituting~\eqref{eq:J}, we obtain 
\begin{equation}\label{eq:barrier-contact}
\begin{aligned}
  Z'(s)&=G(s,h,k_0(s))-k_0'(s)\\
  &=-\frac{(1+s)}{4hs(1-h^2)}\left(h-\frac{1-s}{1+s}\right)
  \bigl((1+s^2)(3h^2+1)+4h(1-s^2)\bigr)<0.
\end{aligned}
\end{equation}
Since $Z(1)>0$, Lemma~\ref{lem:downcrossing} implies $Z(s_0)>0$. The choice of $s_0$ was arbitrary, proving the second inequality.
\end{proof}

\begin{proposition}\label{prop:negative}
If $T_i>0$, $a_i<0$, and $\nu_1\ne\nu_2$, then~\eqref{eq:chord-matching} has at most three distinct solutions.
\end{proposition}
\begin{proof}
Assume $\nu_1>\nu_2$. At a zero of $Y$, put $\kappa=\kappa_1(s)=\kappa_2(s)$. Lemma~\ref{lem:negative-shape-order} gives $h_1(s)<h_2(s)$. We show that $\partial_h\mathcal J(h,s,\kappa)>0$ on $[h_1,h_2]$.

For $s\ge1$, this follows from~\eqref{eq:J-h-large}. For $0<s<1$, set
\begin{equation}\label{eq:Theta}
  \Theta(s,h):=\frac{s\bigl((1-s)(1-h)^3(3h+1)-16h^3\bigr)}{(1-h^2)^2},
\end{equation}
so that
\begin{equation}\label{eq:J-h}
  \partial_h\mathcal J(h,s,\kappa)=\frac{\kappa(\kappa-\Theta(s,h))}{2h^2s}.
\end{equation}
For $0<s<1$ and $0<h<1$,
\begin{align}
  \partial_h\Theta(s,h)
  &=\frac{8hs\bigl((1+h)^3-s(1-h)^3\bigr)}{(h^2-1)^3}<0,\label{eq:Theta-derivative}\\
  k_0(s)-\Theta(s,h_-(s))
  &=\frac{(1-s)\bigl(2s^4+(1-s^3)+(1-s)\bigr)}{2s}>0.\label{eq:Theta-gap}
\end{align}
Lemma~\ref{lem:focus-barrier} therefore yields
\[
  \Theta(s,h)<\Theta(s,h_{-}(s))<k_0(s)<\kappa\qquad(h\in[h_1,h_2]),
\]
and hence $\partial_h\mathcal J>0$ also for $s<1$. By~\eqref{eq:contact-derivative}, $Y'<0$ at every zero of $Y$. Lemma~\ref{lem:downcrossing} and Theorem~\ref{thm:uniform-count} give the bound.
\end{proof}

\subsection{The case \texorpdfstring{$a_1a_2>0$ and $\nu_1=\nu_2$}{a1 a2 > 0 and nu1 = nu2}}\label{sec:same-shape}

Assume $T_1,T_2>0$, $a_1a_2>0$, and $\nu_1=\nu_2$. Set
\[
  \sigma_i:=\operatorname{sgn}a_i,\qquad
  \lambda_i:=\frac{|a_i|}{T_i}>0,\qquad i=1,2.
\]
The rescaling
\[
  \widehat t_i:=T_{i}t,\qquad
  \widehat\xi_i(\widehat t_i):=\frac{T_i^2\xi_i(t)}{|a_i|},\qquad
  \widehat\eta_i(\widehat t_i):=\frac{T_i\eta_i(t)}{|a_i|}=\frac{\eta_i(t)}{\lambda_i},
\]
transforms~\eqref{eq:normal-pair} into
\begin{equation}\label{eq:common-shape-normal}
  \frac{\mathrm{d}\widehat\xi_i}{\mathrm{d}\widehat t_i}=\widehat\xi_i-\widehat\eta_i,\qquad
  \frac{\mathrm{d}\widehat\eta_i}{\mathrm{d}\widehat t_i}=\nu_i\widehat\xi_i-\sigma_i,\qquad \widehat\xi_i<0.
\end{equation}
Denoting by $\widehat P_i$ the half-return map of~\eqref{eq:common-shape-normal}, the relation $\eta_i=\lambda_i\widehat\eta_i$ yields
\begin{equation}\label{eq:half-map-scaling}
  P_i(u)=\lambda_i\widehat P_i(u/\lambda_i),\qquad i=1,2.
\end{equation}

\begin{figure}[htbp]
  \centering
  \includegraphics[width=\linewidth]{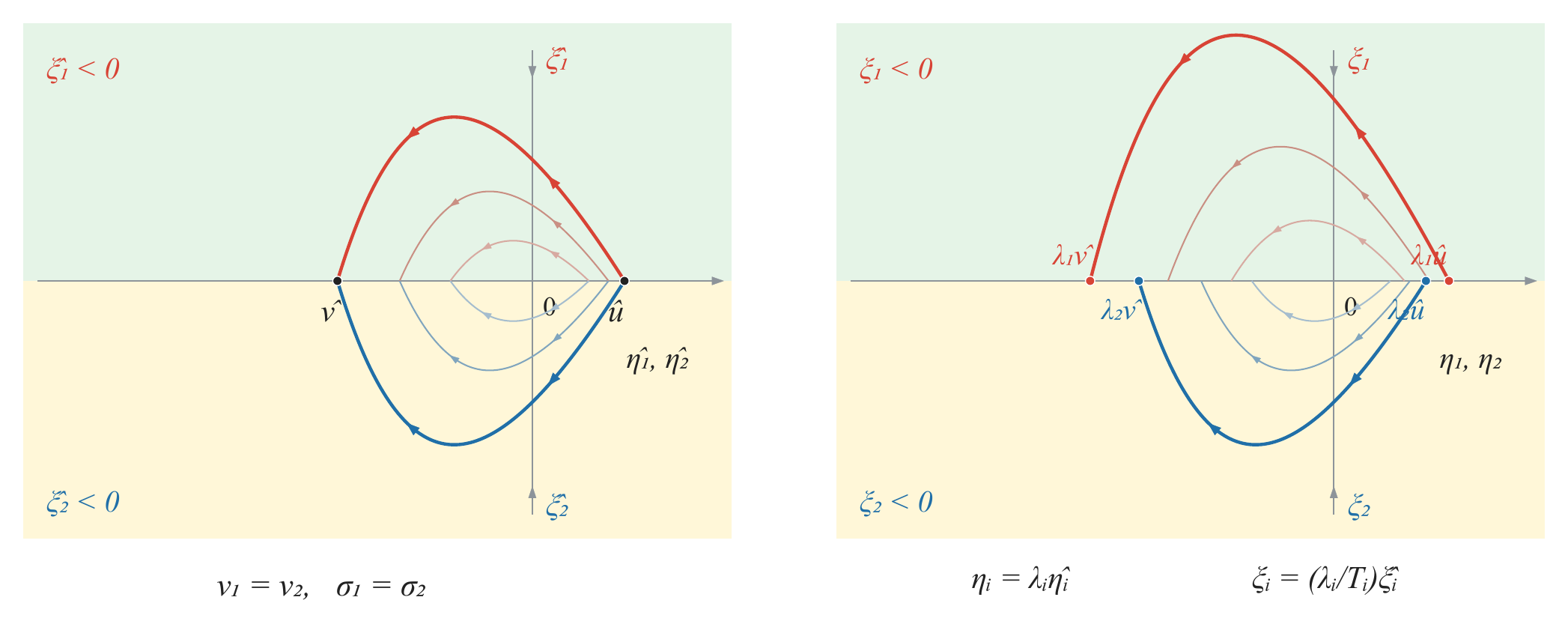}
  \caption{Half-return arcs before (left) and after (right) the coordinate scalings $\eta_i=\lambda_i\widehat\eta_i$ and $\xi_i=(\lambda_i/T_i)\widehat\xi_i$. System $1$ is displayed above the horizontal axis and system $2$ below; the upper copy is reflected for display. The two arcs are compared, not concatenated.}\label{fig:equal-parameter-orbits}
\end{figure}

The rescaled systems coincide because $\sigma_1=\sigma_2$ and $\nu_1=\nu_2$. Write $\widehat P=\widehat P_1=\widehat P_2$ and let $\widehat M$ be its midpoint function. Equation~\eqref{eq:half-map-scaling} gives
\begin{equation}\label{eq:midpoint-scaling}
  M_i(\ell)=\lambda_i\widehat M(\ell/\lambda_i),\qquad \ell\in J_i,\quad i=1,2.
\end{equation}

Differentiating~\eqref{eq:half-map-scaling} shows that, at equal $s$, we have $h_1(s)=h_2(s)$ and $\kappa_1(s)=\kappa_2(s)$. Thus $Y\equiv0$, and Theorem~\ref{thm:uniform-count} does not apply. We instead use~\eqref{eq:midpoint-scaling} to count the matching solutions.

The case $\lambda_1=\lambda_2$ and $\delta=0$ corresponds to the crossing period annuli characterized in~\cite[Theorem~1]{CarmonaEtAl2023Annuli}.

\begin{proposition}\label{prop:same-shape}
Assume $T_1,T_2>0$, $a_1a_2>0$, and $\nu_1=\nu_2$. Then~\eqref{eq:chord-matching} has at most one solution in $J_1\cap J_2$ if $\lambda_1\ne\lambda_2$. If $\lambda_1=\lambda_2$, then~\eqref{eq:chord-matching} has no solution for $\delta\ne0$ and holds for every $\ell\in J_1\cap J_2$ for $\delta=0$ (in which case no periodic orbit is isolated). Consequently,~\eqref{eq:raw-system} has at most one crossing limit cycle.
\end{proposition}

\begin{proof}
Equations~\eqref{eq:concavity} and~\eqref{eq:Mprime} give $\widehat M''<0$. If $\lambda_1\ne\lambda_2$, then~\eqref{eq:midpoint-scaling} implies
\[
  \sgn\BB'(\ell)
  =\sgn\bigl(\widehat M'(\ell/\lambda_1)-\widehat M'(\ell/\lambda_2)\bigr)
  =\sgn(\lambda_1-\lambda_2)\ne0.
\]
Hence $\BB(\ell)=\delta$ has at most one solution. If $\lambda_1=\lambda_2$, then $M_1=M_2$, and the equation reduces to $0=\delta$. The assertion about limit cycles follows from Subsection~\ref{sec:geometry}.
\end{proof}

\begin{figure}[H]
  \centering
  \includegraphics[width=\linewidth]{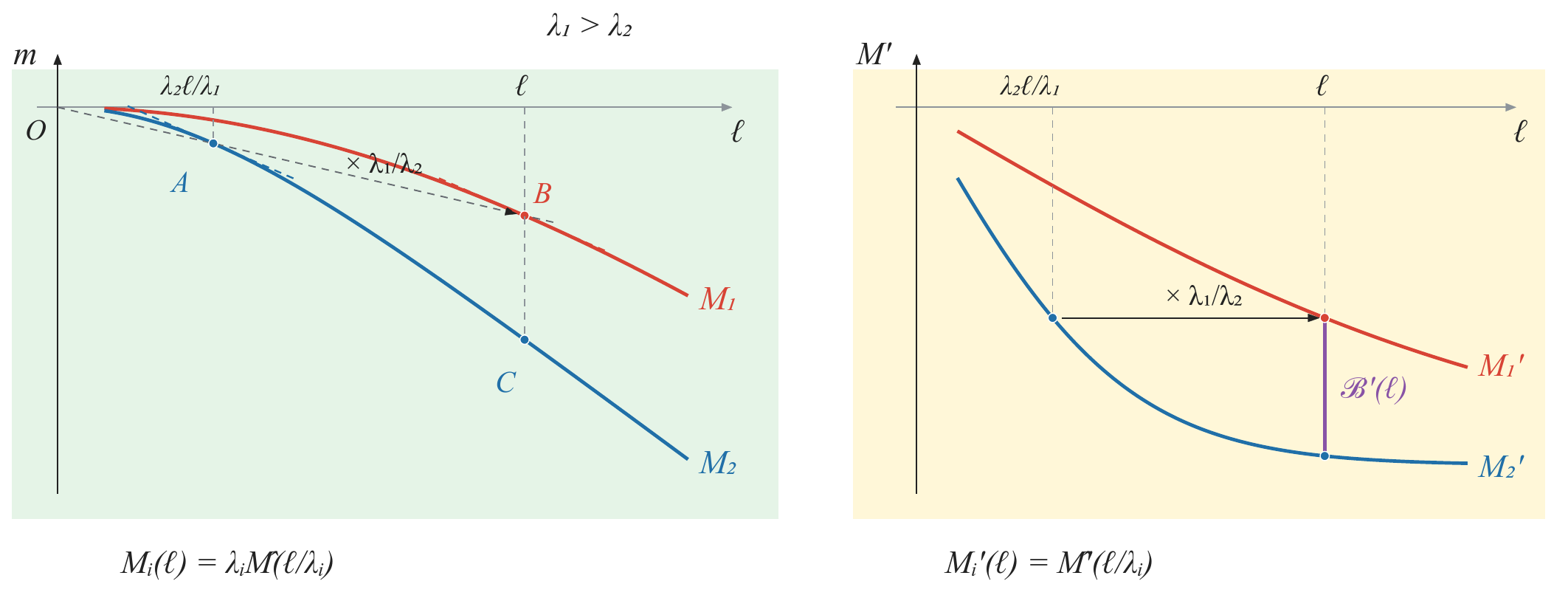}
  \caption{Midpoint functions and their derivatives for $\lambda_1>\lambda_2$. Left: the points $A$ and $B$ are related by a dilation about the origin of factor $\lambda_1/\lambda_2$, and their tangents are parallel. Right: $M_1'(\ell)=M_2'(\lambda_2\ell/\lambda_1)>M_2'(\ell)$, so $\mathcal B'(\ell)>0$.}\label{fig:equal-parameter-midpoints}
\end{figure}

Theorem~\ref{thm:main} now follows from the exhaustive classification in Subsection~\ref{subsec:classification}, Proposition~\ref{prop:concavity-cases}, ~\ref{prop:mixed},~\ref{prop:positive}, ~\ref{prop:negative}, and ~\ref{prop:same-shape}, together with the correspondence in Subsection~\ref{sec:geometry}.

\section*{Acknowledgments}

The second authors are partially supported by National Key R\&D Program of China
[2022YFA1005900], the National Natural Science Foundation of China [12671201,12271354], and Innovation Program of Shanghai Municipal Education Commission [2021-01-07-00-02-E00087]. 

\appendix
\section{Derivation of the third-derivative identity}\label{app:third-derivative}

Fix $u\in I$ and set
\[
\begin{gathered}
  \ell=\frac{u-P(u)}2,\qquad
  m=\frac{u+P(u)}2,\\
  h=-\frac m\ell,\qquad
  s=-P'(u),\qquad
  \kappa=-\ell P''(u).
\end{gathered}
\]
Holding $m,\ell$ fixed, define
\[
  \widetilde P(\xi)=\frac{P(m+\ell\xi)-m}{\ell}.
\]
Then
\[
  \widetilde P(1)=-1,\qquad
  \widetilde P'(1)=-s,\qquad
  \widetilde P''(1)=-\kappa,\qquad
  \widetilde P'''(1)=\ell^2P'''(u).
\]

Set
\[
  w(\xi)=\frac{W(\ell(\xi-h))}{W(u)},\qquad
  \rho=w(-1).
\]
Since $w(1)=1$, formula~\eqref{eq:P-derivatives} gives
\[
  s=\frac{1-h}{1+h}\rho.
\]
Thus, for some coefficient $A$,
\begin{equation}\label{eq:normalized-W}
  \rho=s\frac{1+h}{1-h},\qquad
  w(\xi)=A\xi^2+\frac{1-\rho}{2}\xi+\frac{1+\rho}{2}-A.
\end{equation}

Define, near $\xi=\pm1$,
\[
  R(\xi)=\frac{\xi-h}{w(\xi)},\qquad
  L(\xi)=\frac{R'(\xi)}{R(\xi)}
  =\frac1{\xi-h}-\frac{w'(\xi)}{w(\xi)}.
\]
The rational derivative formula becomes
\[
  \widetilde P'(\xi)
  =\frac{R(\xi)}{R(\widetilde P(\xi))}.
\]
Differentiating once and twice, respectively, and evaluating at $\xi=1$ gives
\begin{align}
  \kappa&=s\bigl(L(1)+sL(-1)\bigr),\label{eq:jet-first}\\
  \ell^2P'''(u)
  &=-\frac{\kappa^2}{s}-s\kappa  L(-1)-sL'(1)+s^3L'(-1).
  \label{eq:jet-second}
\end{align}
Here
\[
  \begin{aligned}
    L(1)&=\frac1{1-h}-2A-\frac{1-\rho}{2},\\
    L(-1)&=-\frac1{1+h}
    -\frac{-2A+(1-\rho)/2}{\rho},\\
    L'(\xi)&=-\frac1{(\xi-h)^2}-\frac{2A}{w(\xi)}
    +\left(\frac{w'(\xi)}{w(\xi)}\right)^2.
  \end{aligned}
\]
Solving~\eqref{eq:jet-first} for $A$ yields
\[
  A=\frac{1+s}{2(1-h)}-\frac{(1+h)\kappa}{4hs}.
\]
Consequently,
\[
  L(1)=-\frac{s+1}{2}+\frac{(1+h)\kappa}{2hs},
  \qquad
  L(-1)=\frac{s+1}{2s}-\frac{(1-h)\kappa}{2hs^2}.
\]
Substitution into~\eqref{eq:jet-second} gives
\[
  \begin{aligned}
    2hs(1-h^2)\ell^2P'''(u)
    ={}&-(3h+1)(1-h^2)\kappa^2\\
    &+s\kappa\bigl[(3h^2+1)(1-s)+4h(s+1)\bigr],
  \end{aligned}
\]
which is~\eqref{eq:J}.
All divisions are valid because $0<h<1$, $s>0$, and
$w(1)=1$, $w(-1)=\rho>0$.

\bibliographystyle{elsarticle-num}
\bibliography{references}
\end{document}